\documentclass[11pt,reqno]{amsart}

\usepackage{amsmath,amssymb,amsthm,mathtools}
\usepackage[margin=1.1in]{geometry}
\usepackage{graphicx}
\usepackage{placeins}
\usepackage{hyperref}
\usepackage[capitalize,noabbrev]{cleveref}
\usepackage{enumitem}

\hypersetup{colorlinks=true,linkcolor=blue,citecolor=blue,urlcolor=blue}

\newtheorem{theorem}{Theorem}[section]
\newtheorem{proposition}[theorem]{Proposition}
\newtheorem{corollary}[theorem]{Corollary}
\newtheorem{lemma}[theorem]{Lemma}

\theoremstyle{definition}
\newtheorem{definition}[theorem]{Definition}
\newtheorem{example}[theorem]{Example}
\newtheorem{problem}[theorem]{Problem}
\theoremstyle{remark}
\newtheorem{remark}[theorem]{Remark}

\newcommand{\C}{\mathbb C}
\newcommand{\R}{\mathbb R}
\newcommand{\CP}{\mathbb {CP}}
\newcommand{\RP}{\mathbb {RP}}
\newcommand{\I}{\mathfrak I}

\renewcommand{\Im}{\operatorname{Im}}

\title[Inflection curves of rational vector fields]{Inflection curves of rational vector fields}

\author[B. Shapiro]{Boris Shapiro}
\address{Department of Mathematics, Stockholm University, SE-106 91 Stockholm, Sweden}
\email{shapiro@math.su.se}

\author[G. Tahar]{Guillaume Tahar}
\address{Beijing Institute of Mathematical Sciences and Applications, Huairou District, Beijing, China}
\email{guillaume.tahar@bimsa.cn}

\date{\today}

\subjclass[2020]{Primary 37F75, 34M45; Secondary 14P25, 30F99, 58K05}
\keywords{rational vector field, inflection curve, real algebraic curve, separatrix, rational map, dessin d'enfant, exact differential}

\begin{document}

\begin{abstract}
We initiate the study of inflection curves of rational vector fields on the Riemann sphere.  For a rational vector field
\[
        v_R=-R(z)\frac{\partial}{\partial z},
        \qquad R(z)=\frac{Q(z)}{P(z)},
\]
we define its affine regular inflection locus by
\[
        \{z\in \C: R(z)\ne0,\ P(z)\ne0,\ \Im R'(z)=0\}
\]
and its algebraic inflection curve by the closure of this locus, equivalently by
\[
        \I_R=(R')^{-1}(\RP^1).
\]
We prove an explicit defining equation, degree bounds, local normal forms near poles, the asymptotic directions at infinity, and a maximum-principle obstruction excluding compact components without poles.  We also explain that these curves are precisely the real dessins associated with exact rational differentials, i.e. rational differentials with zero residues.  Finally, we give a reducibility criterion for the complexification, prove a generic irreducibility statement in the usual separated-variable sense, and classify the exact dessins of degree at most two.  The remaining low-degree degree-three problem is reduced to three explicit normal forms.
\end{abstract}

\maketitle

\section{Introduction}

Let
\[
        R(z)=\frac{Q(z)}{P(z)}
\]
be a nonconstant rational function, where $P,Q\in\C[z]$ are coprime.  We consider the associated rational vector field
\[
        v_R=-R(z)\frac{\partial}{\partial z}
\]
on $\CP^1$.  Away from the zeros and poles of $R$, its real trajectories are ordinary plane curves.  It is natural to ask for the locus where these trajectories have inflection points.

The elementary observation underlying the paper is that this locus is governed by the rational derivative $R'$:
\[
        \Im R'(z)=0.
\]
Thus the inflection curve is both a real algebraic plane curve and the inverse image of the real projective line under the rational map $R'$:
\[
        \I_R=(R')^{-1}(\RP^1).
\]
This makes the object accessible from two complementary viewpoints: local differential geometry of vector-field trajectories and the topology of real dessins.

The study of meromorphic and rational vector fields on the Riemann sphere has a substantial literature; see, for example, \cite{HajekI,HajekII,Benzinger,NeedhamKing,BrannerDias,DiasGarijo,Klimes,Fiedler}.  For Newton graphs and related separatrix combinatorics see \cite{STW}; for dessins and quadratic differentials see \cite{LandoZvonkin,Strebel}.  Wronskians and real Wronski maps enter through the algebraic equation below; see \cite{EG}.  The present paper grew out of considerations appearing in \cite{AHS}, where rational vector fields occur in connection with invariant sets of first-order differential operators.

Our main results are as follows.
\begin{enumerate}[label=(\roman*)]
\item We derive the defining equation
\[
        \frac{W(z)\overline{P(z)}^2-\overline{W(z)}P(z)^2}{2i}=0,
        \qquad W=Q'P-QP',
\]
and obtain degree bounds (Corollary~\ref{cor:degree-bound}).
\item We prove the local pole model: if $R$ has a pole of order $s$ at $z_0$, then $\I_R$ has $s+1$ smooth analytic branches through $z_0$, tangent to the separatrix directions of $v_R$ (Theorem~\ref{thm:local-pole}).
\item We determine the asymptotic directions at infinity and prove that every bounded connected component of the affine curve contains a pole of $R$. In particular, polynomial vector fields have no bounded affine inflection components (Theorem~\ref{thm:asymptotic}).
\item We show that the inflection curves of rational vector fields are precisely the real dessins of exact rational differentials (Corollary~\ref{cor:exact-dessins}).
\item We give a separated-variable irreducibility criterion for the complexification of $\I_R$ (Corollary~\ref{cor:generic-irreducible}) and solve the low-degree problem for $\deg R'\le2$ (Section~\ref{sec:LowDegree}).
\end{enumerate}

\section{Inflection points of analytic vector fields}

We first recall the elementary differential-geometric calculation behind the definition.

Let $W=(u,v)$ be a smooth real vector field in a domain of $\R^2$.  A regular point $(x,y)$, i.e. a point where $W(x,y)\ne0$, is an inflection point of the trajectory through $(x,y)$ precisely when the acceleration vector is parallel to the velocity vector.  Equivalently,
\begin{equation}\label{eq:general-inflection}
        u\,(u v_x+v v_y)-v\,(u u_x+v u_y)=0.
\end{equation}

\begin{lemma}\label{lem:analytic-inflection}
Let $R$ be holomorphic in a domain $U\subset\C$, and write $R=u+iv$.  At every point of $U$ where $R(z)\ne0$, the trajectory of the vector field $R(z)\partial_z$ has an inflection if and only if
\[
        \Im R'(z)=0.
\]
The same condition holds for the vector field $-R(z)\partial_z$.
\end{lemma}

\begin{proof}
In complex notation the velocity is $R(z)$, while the acceleration along a trajectory is
\[
        \frac{d}{dt}R(z(t))=R'(z(t))R(z(t)).
\]
The two real vectors represented by complex numbers $a$ and $b$ are parallel if and only if $\Im(\overline a b)=0$.  Hence the inflection condition is
\[
        \Im\bigl(\overline{R(z)}R'(z)R(z)\bigr)=|R(z)|^2\Im R'(z)=0.
\]
Since $R(z)\ne0$, this is equivalent to $\Im R'(z)=0$.  Replacing $R$ by $-R$ does not change the condition.
\end{proof}

\begin{definition}\label{def:inflection-curve}
Let $R$ be rational.  The \emph{regular affine inflection locus} is
\[
        \I_R^{\rm reg}=\{z\in\C: R(z)\ne0,\ P(z)\ne0,\ \Im R'(z)=0\}.
\]
The \emph{affine algebraic inflection curve} $\I_R$ is the closure of this locus in $\C$ after clearing denominators.  Its projective closure in $\RP^2$ will be denoted by $\overline{\I}_R$.
\end{definition}

Zeros of $R$ are singular points of the vector field.  They are not relevant for the curvature of regular trajectories, but they may lie on $\I_R$ if $\Im R'(z)=0$.  Poles of $R$, by contrast, are part of the natural algebraic closure after clearing denominators and give distinguished singularities of $\I_R$.

\section{Algebraic equation and degree}\label{sec:Degree}

Let
\[
        R(z)=\frac{Q(z)}{P(z)},
        \qquad (P,Q)=1,
\]
and set
\[
        W(z)=Q'(z)P(z)-Q(z)P'(z).
\]
Then
\[
        R'(z)=\frac{W(z)}{P(z)^2}.
\]

\begin{proposition}\label{prop:equation}
The affine inflection curve $\I_R$ is defined by the real polynomial equation
\begin{equation}\label{eq:defining-polynomial}
        F_R(z,\bar z):=\frac{W(z)\overline{P(z)}^{2}-\overline{W(z)}P(z)^2}{2i}=0.
\end{equation}
In real coordinates $z=x+iy$, this is a real polynomial equation.
\end{proposition}

\begin{proof}
Away from the poles of $R$, the condition $\Im R'(z)=0$ is equivalent to
\[
        \frac{W(z)}{P(z)^2}=\frac{\overline{W(z)}}{\overline{P(z)}^2}.
\]
Multiplication by $P(z)^2\overline{P(z)}^2$ gives \eqref{eq:defining-polynomial}.  The numerator is anti-invariant under complex conjugation, hence division by $2i$ gives a real-valued polynomial.  Since $z$ and $\bar z$ are polynomial expressions in $x,y$, $F_R$ is a real polynomial in $x,y$.
\end{proof}

\begin{corollary}\label{cor:degree-bound}
Let $k=\deg Q$, $\ell=\deg P$, and $n=\deg W$.  Then $\I_R$ is contained in a real algebraic curve of degree at most
\[
        n+2\ell.
\]
For generic pairs $(P,Q)$ with $k\ne \ell$, one has $n=k+\ell-1$, and the degree of $F_R$ is $k+3\ell-1$.  If $k=\ell$, the leading terms cancel in the Wronskian, and generically $n=2\ell-2$, so the corresponding degree is $4\ell-2$.
\end{corollary}

\begin{proof}
The degree of $W(z)\overline{P(z)}^2$ as a polynomial in $x,y$ is at most $n+2\ell$, and the same is true for its conjugate.  The generic degree of $W$ follows by comparing the leading terms of $Q'P$ and $QP'$.
\end{proof}

\begin{remark}\label{rem:poles-in-equation}
The equation \eqref{eq:defining-polynomial} includes the poles of $R$.  Thus it is better regarded as the natural algebraic closure of the regular inflection locus.  This convention is useful because poles are exactly the points where the curve has the distinguished singularities described in the next section.
\end{remark}

\section{Local structure near poles and ordinary singularities}

\begin{theorem}\label{thm:local-pole}
Suppose that $z_0$ is a pole of $R$ of order $s\ge1$.  Write
\[
        R(z)=\frac{a}{(z-z_0)^s}+O((z-z_0)^{-s+1}),
        \qquad a\ne0.
\]
Then near $z_0$, the inflection curve $\I_R$ has $s+1$ smooth real analytic branches through $z_0$, with tangent rays determined by
\begin{equation}\label{eq:pole-directions}
        \arg(z-z_0)=\frac{\arg(-sa)+m\pi}{s+1},
        \qquad m=0,1,\ldots,2s+1.
\end{equation}
Opposite rays in \eqref{eq:pole-directions} belong to the same analytic branch.  In particular, if $z_0$ is a simple pole of $R$, then $z_0$ is an ordinary real node of $\I_R$.
\end{theorem}

\begin{proof}
Put $w=z-z_0$.  Then
\[
        R'(z)=-sa w^{-s-1}+O(w^{-s}).
\]
The equation $\Im R'(z)=0$ is equivalent, after multiplication by $|w|^{2s+2}$, to
\[
        \Im\bigl((-sa)\overline w^{s+1}+O(|w|^{s+2})\bigr)=0.
\]
The leading homogeneous part is $\Im((-sa)\overline w^{s+1})$.  Writing $w=re^{i\theta}$, its zero set is
\[
        \arg(-sa)-(s+1)\theta\in \pi\mathbb Z,
\]
which is equivalent to \eqref{eq:pole-directions}.  The leading homogeneous polynomial has $2s+2$ distinct real linear factors.  The standard real-analytic factorization of plane curve germs with distinct tangent cone therefore gives $s+1$ smooth analytic branches.
\end{proof}

\begin{proposition}\label{prop:singularities}
Away from the poles of $R$, singular points of the real analytic curve $\I_R$ can occur only at points $z$ satisfying
\[
        R''(z)=0,
        \qquad
        \Im R'(z)=0.
\]
Consequently, for a generic rational function $R$, no such singularities occur away from the poles.
\end{proposition}

\begin{proof}
Away from the poles, $\I_R$ is the zero set of the harmonic function $h(z)=\Im R'(z)$.  A point of this zero set is singular only if $dh=0$.  Since $h$ is the imaginary part of the holomorphic function $R'$, the condition $dh=0$ is equivalent to $R''(z)=0$.  Thus the stated conditions are necessary.

The condition that a critical value of the rational map $R'$ lie on $\RP^1$ is real-codimension one in the parameter space.  Hence for generic $R$, all critical values of $R'$ avoid $\RP^1$ except those forced by poles, and the stated singularities do not occur.
\end{proof}

\section{Asymptotic directions and compact components}

The behavior of $\I_R$ at infinity is governed by the Laurent expansion of $R'$ at infinity.

\begin{theorem}\label{thm:asymptotic}
Assume that near infinity
\[
        R'(z)=c z^m+O(z^{m-1}),
        \qquad c\ne0,
\]
where $m\in\mathbb Z$.  If $m\ne0$, then outside a sufficiently large disk the curve $\I_R$ consists of $2|m|$ smooth unbounded arcs asymptotic to the rays
\begin{equation}\label{eq:asymptotic-rays}
        \arg z=\frac{j\pi-\arg c}{m},
        \qquad j\in\mathbb Z,
\end{equation}
where the right-hand side is taken modulo $2\pi$.  If $m=0$, then $R'(z)\to c$.  In this case $\I_R$ is bounded whenever $\Im c\ne0$.
\end{theorem}

\begin{proof}
For $z=re^{i\theta}$, the leading term gives
\[
        \Im R'(z)=r^m\Im(c e^{im\theta})+O(r^{m-1})
\]
when $m>0$.  If $m<0$, multiply the equation by $r^{-m}$; the same leading angular equation is obtained.  Thus the asymptotic directions are determined by
\[
        \Im(c e^{im\theta})=0,
\]
which gives \eqref{eq:asymptotic-rays}.  Since the zeros of this trigonometric polynomial are simple, the implicit-function theorem gives exactly one smooth branch approaching each asymptotic ray for sufficiently large $r$.

If $m=0$, then $R'(z)=c+O(1/z)$.  If $\Im c\ne0$, then $\Im R'(z)$ has the same nonzero sign as $\Im c$ for all sufficiently large $|z|$.  Therefore $\I_R$ has no points near infinity and is bounded.
\end{proof}

\begin{corollary}\label{cor:degree-cases}
Let $k=\deg Q$ and $\ell=\deg P$.  If $k\ne\ell$, then generically
\[
        R'(z)=c z^{k-\ell-1}+O(z^{k-\ell-2}),
        \qquad c\ne0.
\]
Thus $\I_R$ has $2|k-\ell-1|$ ends if $k-\ell-1\ne0$.  If $k=\ell+1$, then $R'(z)$ tends to a nonzero constant, and $\I_R$ is compact whenever the imaginary part of this constant is nonzero.
\end{corollary}

\begin{remark}
When $k=\ell$, the leading term of $R$ at infinity is constant, and the leading nonconstant term must be inspected.  Generically $R(z)=a+b/z+O(z^{-2})$ with $b\ne0$, so $R'(z)=-bz^{-2}+O(z^{-3})$.  Hence $\I_R$ has four ends in the generic case $k=\ell$.
\end{remark}

\begin{proposition}\label{prop:no-compact-without-pole}
Every bounded connected component of the affine curve $\I_R$ contains at least one finite pole of $R$.
\end{proposition}

\begin{proof}
Let $K$ be a bounded connected component of $\I_R$, and suppose that $K$ contains no pole of $R$.  Choose a bounded open neighborhood $U$ of $K$ whose closure contains no pole and such that $\partial U\cap\I_R=\varnothing$.  In $U$ the function
\[
        h(z)=\Im R'(z)
\]
is harmonic and nonzero on $\partial U$.

The zero set of a nontrivial harmonic function is a locally finite real-analytic graph.  At a regular point of this graph the valency is two; at a critical zero the leading homogeneous harmonic polynomial has degree at least two, hence the valency is at least four and is even.  Since $K$ is compact, it is a finite graph and all its vertices have even valency.  Therefore, if $K$ is nonempty, it contains a cycle.  This cycle bounds a disk $D\subset U$ on whose boundary $h=0$.  By the maximum principle $h\equiv0$ on $D$, and hence $R'$ is constant real on $D$.  Analytic continuation then implies that $R'$ is constant on $\CP^1$, contrary to the assumption that $K$ is a genuine component of the inflection curve of a nonconstant derivative.  Thus $K$ must contain a pole.
\end{proof}

\begin{corollary}\label{cor:polynomial-no-oval}
If $R$ is a polynomial, then every connected component of the affine inflection curve $\I_R$ is unbounded.  More generally, if all finite poles of $R$ are simple and $\I_R$ is nonsingular away from them, then the number of bounded connected components of $\I_R$ is at most the number of poles of $R$, and hence at most $\deg P$.
\end{corollary}

\begin{proof}
The first claim is the special case of \cref{prop:no-compact-without-pole} in which $R$ has no finite poles.  The second follows because distinct connected components are disjoint and each bounded component contains at least one pole.
\end{proof}

\section{Relation with separatrices}

The local directions of $\I_R$ at poles are closely related to the separatrix directions of the rational vector field.  Separatrix graphs for rational vector fields on the sphere were studied in detail in \cite{Benzinger,NeedhamKing,BrannerDias,DiasGarijo,Klimes,Fiedler}.

If
\[
        R(z)=\frac{a}{(z-z_0)^s}+O((z-z_0)^{-s+1}),
        \qquad a\ne0,
\]
then the vector field $-R(z)\partial_z$ has $s+1$ incoming and $s+1$ outgoing separatrix directions at $z_0$, alternating around $z_0$.  These directions are obtained by solving the leading differential equation
\[
        \dot w=-a w^{-s},
        \qquad w=z-z_0.
\]

\begin{proposition}\label{prop:sep-inflection}
At a pole $z_0$ of order $s$, the tangent rays of $\I_R$ coincide with the union of the incoming and outgoing separatrix directions of the vector field $-R(z)\partial_z$.
\end{proposition}

\begin{proof}
For the leading equation $\dot w=-a w^{-s}$, one has
\[
        \frac{d}{dt}w^{s+1}=-(s+1)a.
\]
Thus separatrix rays are characterized by the condition that $w^{s+1}/(-a)$ be real.  Equivalently,
\[
        \arg w=\frac{\arg(-a)+j\pi}{s+1}.
\]
On the other hand, by \cref{thm:local-pole}, the tangent rays of $\I_R$ are determined by
\[
        \Im((-sa)w^{-s-1})=0,
\]
or, equivalently, by the same condition up to the harmless positive factor $s$.
\end{proof}

\section{Exact real dessins and derivative-realizability}

The previous sections used only the fact that $f=R'$ is a rational map.  We now record the extra restriction imposed by being a rational derivative.

For a rational map $f$, let
\[
        \Gamma_f=f^{-1}(\RP^1)\subset \CP^1.
\]
This is a real dessin in the broad sense: it is the pullback of a circle under a branched cover of the sphere.  Inflection curves correspond to the special case $f=R'$.

\begin{proposition}\label{prop:exact}
A rational function $f$ can be represented in the form
\[
        f=R'
\]
for some rational function $R$ if and only if all residues of the meromorphic differential
\[
        f(z)\,dz
\]
at finite poles vanish.
\end{proposition}

\begin{proof}
If $f=R'$, then $f(z)\,dz=dR$, and exact meromorphic differentials have zero residues.

Conversely, write the partial fraction decomposition of $f$ as
\[
        f(z)=p(z)+\sum_{a}\sum_{j=1}^{m_a}\frac{c_{a,j}}{(z-a)^j}.
\]
The residue of $f(z)\,dz$ at $a$ is $c_{a,1}$.  If all residues vanish, then
\[
        R(z)=\int p(z)\,dz-
        \sum_a\sum_{j=2}^{m_a}\frac{c_{a,j}}{(j-1)(z-a)^{j-1}}
\]
is a rational primitive of $f$.  Hence $f=R'$.
\end{proof}

\begin{corollary}\label{cor:exact-dessins}
Inflection curves of rational vector fields are precisely the real dessins associated with exact rational differentials.
\end{corollary}

\begin{corollary}\label{cor:pole-restriction}
If $f=R'$, then every finite pole of $f$ has order at least two.  A pole of $f$ of order $m$ gives a vertex of $f^{-1}(\RP^1)$ with $m$ smooth analytic branches, equivalently $2m$ incident tangent rays.
\end{corollary}

\begin{proof}
The first assertion is exactly the vanishing of the residue coefficient.  The second is \cref{thm:local-pole} applied with $m=s+1$.
\end{proof}

This gives a first answer to the derivative-realizability problem: not every real dessin can arise from a rational vector field, because finite vertices lying over $\infty\in\RP^1$ must have even valency at least four and must satisfy the zero-residue condition in local coordinates.

\section{Irreducibility of the complexification}

Let
\[
        f(z)=R'(z)=\frac{A(z)}{B(z)}
\]
be written in lowest terms.  Denote by
\[
        f^\sigma(w)=\overline{f(\bar w)}
        =\frac{A^\sigma(w)}{B^\sigma(w)}
\]
the coefficientwise conjugate rational function.  The complexification of the equation $\Im f(z)=0$ is the separated-variable curve
\begin{equation}\label{eq:complexification}
        C_f:\quad A(z)B^\sigma(w)-A^\sigma(w)B(z)=0
        \subset \CP^1_z\times\CP^1_w.
\end{equation}
The real curve $\overline{\I}_R$ is obtained from $C_f$ by imposing the antiholomorphic real structure $w=\bar z$.

\begin{proposition}\label{prop:irreducibility-criterion}
The complex curve $C_f$ is irreducible if and only if the two covers
\[
        f:\CP^1_z\to\CP^1,
        \qquad
        f^\sigma:\CP^1_w\to\CP^1
\]
have no nontrivial common subcover.  Equivalently, the function fields $\C(z)$ and $\C(w)$, considered as extensions of $\C(t)$ by $t=f(z)=f^\sigma(w)$, have no common intermediate field other than $\C(t)$.
\end{proposition}

\begin{proof}
The curve $C_f$ is the fiber product of the two covers $f$ and $f^\sigma$.  Reducible components of this fiber product are exactly the decompositions of the tensor product
\[
        \C(z)\otimes_{\C(t)}\C(w).
\]
The standard separated-variable criterion of Fried and MacRae identifies such decompositions with common intermediate covers of the two maps; see \cite{FriedMacRae}.  This gives the stated equivalence.
\end{proof}

\begin{corollary}\label{cor:generic-irreducible}
Outside the proper algebraic locus where $f=R'$ and $f^\sigma$ have a common subcover, the complexification $C_f$ is irreducible. In particular, irreducibility holds for a Zariski open dense subset of each exact degree class for which a full symmetric monodromy member exists.
\end{corollary}

\begin{proof}
Having a common subcover of degree $>1$ is an algebraic condition on the coefficients of the two rational maps; it is the usual exceptional locus in the separated-variable problem.  Therefore its complement is Zariski open. When one member of the exact degree class has full symmetric monodromy and is not equivalent to its conjugate cover, the exceptional locus is proper.  Generic polynomial derivatives already provide such members, by the classical full-monodromy property of a general polynomial of given degree.  The result follows from Proposition~\ref{prop:irreducibility-criterion}.
\end{proof}

\begin{remark}
The statement is deliberately formulated for the complexification.  The real affine set $\I_R\subset\C$ may have several connected components even when $C_f$ is irreducible.  For instance, the equation $2xy+1=0$ is irreducible but has two affine branches.
\end{remark}

\section{Low-degree exact dessins}\label{sec:LowDegree}

We now solve the first cases of the low-degree classification problem.  Here $\deg f$ means the degree of the rational map $f=R'$.

\begin{proposition}\label{prop:degree-one-two}
Let $f=R'$ be an exact rational function of degree at most two.
\begin{enumerate}[label=(\alph*)]
\item If $\deg f=1$, then $f$ is a nonconstant polynomial of degree one.  Hence $\I_R$ is an affine real line.
\item If $\deg f=2$, then exactly one of the following occurs.
\begin{enumerate}[label=(\roman*)]
\item $f$ is a quadratic polynomial.  After translation and complex scaling of the source coordinate it has the form
\[
        f(z)=\alpha z^2+\beta,
        \qquad \alpha\ne0.
\]
The inflection curve is affinely equivalent to
\[
        2xy+\Im\beta=0.
\]
Thus it is a smooth hyperbola if $\Im\beta\ne0$, and the union of two transverse lines if $\Im\beta=0$.
\item $f$ has one finite pole, necessarily of order two.  Then
\[
        f(z)=\beta+\frac{\alpha}{(z-p)^2},
        \qquad \alpha\ne0.
\]
In the coordinate $u=1/(z-p)$ the inflection curve is again affinely equivalent to
\[
        2xy+\Im\beta=0.
\]
In the $z$-sphere, the point $p$ is an ordinary real node of the projective dessin.
\end{enumerate}
\end{enumerate}
\end{proposition}

\begin{proof}
If $\deg f=1$ and $f$ had a finite pole, that pole would be simple, contradicting \cref{cor:pole-restriction}.  Therefore $f$ is linear, and $\Im f(z)=0$ is a real affine line.

Let $\deg f=2$.  The pole divisor of a degree-two rational function has total degree two.  Since finite poles of an exact rational function cannot be simple, either all pole order is at infinity, in which case $f$ is a quadratic polynomial, or there is one finite pole of order two and no pole at infinity, in which case $f(z)=\beta+\alpha/(z-p)^2$.  These are the two alternatives.

In the polynomial case, completing the square gives $f(z)=\alpha(z-p)^2+\beta$.  Replacing $z-p$ by a square root of $\alpha$ reduces the equation to $\Im(u^2+\beta)=0$, i.e.
\[
        2xy+\Im\beta=0,
        \qquad u=x+iy.
\]
The finite-pole case is identical after the inversion $u=1/(z-p)$.  The final assertion about the node at $p$ is the case $s=1$ of \cref{thm:local-pole}.
\end{proof}

\begin{corollary}\label{cor:degree-two-irreducible}
For $\deg R'\le2$, the projective complexification of $\I_R$ is irreducible except in the evident case where the unique finite critical value of $R'$ lies on $\RP^1$; in that exceptional case the curve is the union of two real lines after a suitable affine or inverted affine coordinate.
\end{corollary}

\begin{proof}
The normal forms in \cref{prop:degree-one-two} reduce the claim to the equation $2xy+c=0$.  If $c\ne0$, it is irreducible.  If $c=0$, it factors as the union of the two coordinate axes.
\end{proof}

The next degree already displays the first genuinely nontrivial dessin combinatorics, but exactness still gives a short list of normal forms.

\begin{proposition}\label{prop:degree-three-normal-forms}
Let $f=R'$ be an exact rational function of degree three.  Up to translation of the source coordinate, $f$ belongs to one of the following three classes:
\begin{enumerate}[label=(\roman*)]
\item a cubic polynomial;
\item one finite triple pole and no pole at infinity,
\[
        f(z)=c+\frac{a}{z^3}+\frac{b}{z^2},
        \qquad a\ne0;
\]
\item one finite double pole and a simple pole at infinity,
\[
        f(z)=az+b+\frac{c}{z^2},
        \qquad ac\ne0.
\]
\end{enumerate}
For a generic member of each class, all singularities of $f^{-1}(\RP^1)$ are the pole-vertices forced by \cref{cor:pole-restriction}; additional finite singularities occur exactly when a remaining critical value of $f$ lies on $\RP^1$.
\end{proposition}

\begin{proof}
The pole divisor of a degree-three rational map has total degree three.  Finite poles of an exact rational function have order at least two.  Therefore the only possible pole partitions are: all order at infinity, giving a cubic polynomial; one finite pole of order three; or one finite pole of order two together with a simple pole at infinity.  These alternatives give the displayed forms after translating the finite pole to the origin and using the zero-residue condition to remove the $1/z$ term.

The assertion about singularities follows from \cref{prop:singularities}: away from poles, singularities of the real dessin occur only at critical points whose critical values are real.
\end{proof}

\begin{remark}
Thus the remaining degree-three classification is finite-dimensional and explicit: one has only to follow the four critical values of a degree-three rational map relative to the circle $\RP^1$, with the pole critical values forced to lie over $\infty$.  This is a tractable continuation problem rather than an unrestricted Hilbert-sixteenth-type problem.
\end{remark}

\section{Illustrations}

The following two figures are not meant to be special normal forms.  They are included to make visible the geometry behind the preceding local statements.  In both figures the thin gray curves are real trajectories of the displayed vector field, the red curve is the algebraic inflection curve $\I_R=\{\Im R'(z)=0\}$, black dots denote poles of $R$, and open circles denote zeros of $R$.

\begin{figure}[t]
\centering
\includegraphics[width=.92\textwidth]{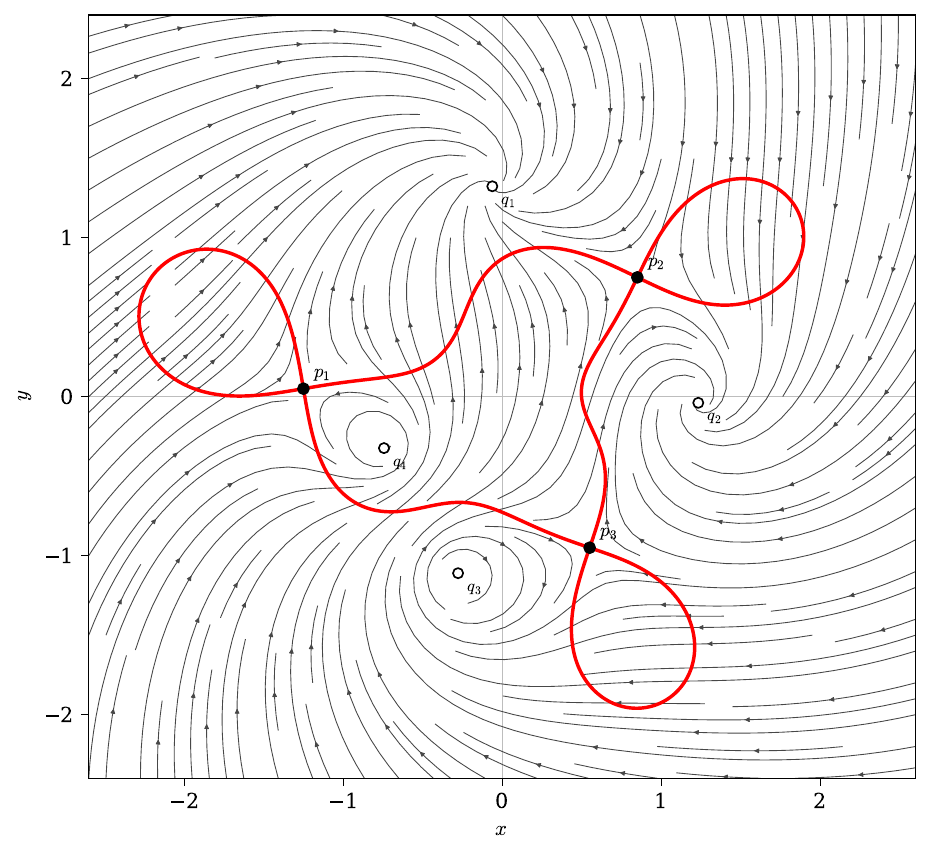}
\caption[Compact three-pole example]{A compact three-pole example for the rational vector field $v_{R_1}=-R_1(z)\partial_z$, where
$R_1(z)=(0.75+0.90i)z+\frac{1+0.35i}{z-p_1}+\frac{-0.75+0.95i}{z-p_2}+\frac{0.85-0.65i}{z-p_3}$, with
$p_1=-1.25+0.05i$, $p_2=0.85+0.75i$, and $p_3=0.55-0.95i$.  Since $R_1'(z)\to0.75+0.90i$ at infinity, the inflection curve is bounded, in agreement with \Cref{thm:asymptotic}.}
\label{fig:compact-three-poles}
\end{figure}

\begin{figure}[t]
\centering
\includegraphics[width=.92\textwidth]{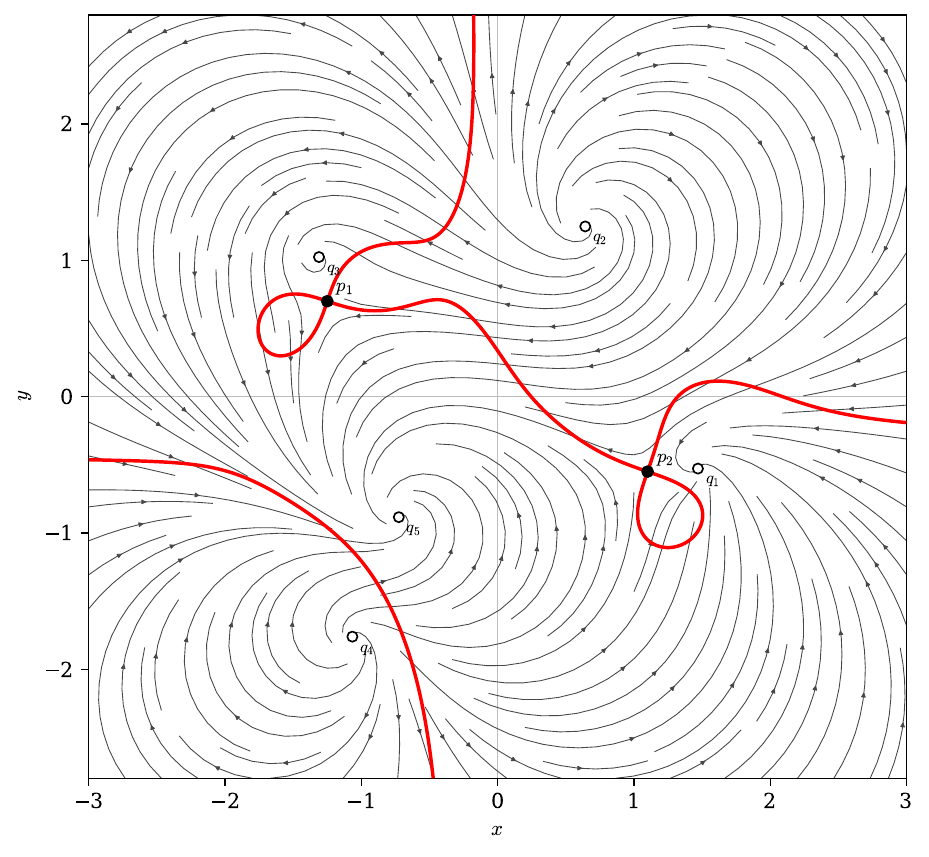}
\caption[Noncompact example with polynomial drift]{A noncompact example for the rational vector field $v_{R_2}=-R_2(z)\partial_z$, where
$R_2(z)=\frac13z^3+(0.28+0.35i)z^2+(0.40-0.60i)z+\frac{0.80-0.60i}{z-p_1}+\frac{-0.90+0.70i}{z-p_2}$, with
$p_1=-1.25+0.70i$ and $p_2=1.10-0.55i$.  The quadratic leading term of $R_2'$ produces unbounded ends of $\I_{R_2}$, while the two simple poles appear as ordinary nodes on the red curve.}
\label{fig:noncompact-two-poles}
\end{figure}

\FloatBarrier

\section{Examples}

\begin{example}[Polynomial vector fields]
Let $R(z)=z^d$, $d\ge2$.  Then
\[
        R'(z)=d z^{d-1},
        \qquad
        \I_R=\{z:\Im z^{d-1}=0\}.
\]
This is the union of $d-1$ real lines through the origin.  This is nongeneric: the critical value $0$ of $R'$ lies on $\RP^1$.
\end{example}

\begin{example}[A simple pole]
Let $R(z)=1/z$.  Then $R'(z)=-z^{-2}$, so
\[
        \I_R=\{z:\Im(-z^{-2})=0\}.
\]
After clearing denominators this is equivalent to $\Im z^2=0$.  Thus the pole is an ordinary node, in agreement with \cref{thm:local-pole}.
\end{example}

\begin{example}[Compact case]
Let
\[
        R(z)=(1+i)z+\frac{a_1}{z-p_1}+\cdots+\frac{a_\ell}{z-p_\ell},
\]
where the $p_j$'s are distinct.  Then
\[
        R'(z)=1+i-\sum_{j=1}^{\ell}\frac{a_j}{(z-p_j)^2}.
\]
Since $R'(z)\to1+i$, \cref{thm:asymptotic} implies that $\I_R$ is compact.  Each simple pole $p_j$ is an ordinary node, provided no additional singularities occur.
\end{example}

\begin{example}[Noncompact case]
If
\[
        R(z)=z+\frac{1}{z},
\]
then
\[
        R'(z)=1-z^{-2}.
\]
The limiting constant at infinity is real, so the compactness criterion does not apply.  Indeed, the equation $\Im(1-z^{-2})=0$ is equivalent after clearing denominators to $\Im z^2=0$, and the inflection curve is unbounded.
\end{example}

\section{Open problems}

The results above settle the basic local theory, the exactness obstruction, generic irreducibility in the separated-variable sense, and the first low-degree cases.  Several natural questions remain.

\begin{problem}\label{prob:hilbert}
Fix $(k,\ell)$.  Classify the topological types of real algebraic curves which occur as $\I_R$ for rational functions
\[
        R=Q/P,
        \qquad \deg Q=k,
        \qquad \deg P=\ell,
\]
with $P,Q$ coprime and generic.
\end{problem}

\begin{problem}\label{prob:dessin}
Classify real dessins associated with exact rational differentials, i.e. embedded graphs of the form
\[
        f^{-1}(\RP^1),
        \qquad f(z)\,dz=dR.
\]
Equivalently, classify real dessins satisfying the zero-residue condition at all finite poles lying over $\infty\in\RP^1$.
\end{problem}

\begin{problem}\label{prob:degree-three}
Complete the degree-three classification using the three normal forms in \cref{prop:degree-three-normal-forms}.  In particular, determine which relative positions of the non-forced critical values with respect to $\RP^1$ give distinct embedded graphs.
\end{problem}

\begin{problem}\label{prob:components}
For a fixed pole divisor of $R$, determine the sharp upper bound for the number of bounded connected components of $\I_R$.  By \cref{prop:no-compact-without-pole}, every such component must contain at least one finite pole; the question is when this trivial bound is sharp.
\end{problem}

\begin{problem}\label{prob:monodromy}
Describe the monodromy of the covering
\[
        R':\CP^1\to\CP^1
\]
which is compatible with exactness.  In particular, determine how the zero-residue condition restricts the cyclic orders of half-edges in the real dessin $\I_R=(R')^{-1}(\RP^1)$.
\end{problem}

\section*{Acknowledgements}
The first author thank Beijing Institute for Mathematical Sciences and Applications for hospitality in Spring 2024.

\end{document}